\documentclass[12pt]{article}
\input epsf.tex

\usepackage{graphicx}
\usepackage{amsmath,amsthm,amsfonts,amscd,amssymb,comment,eucal,latexsym,mathrsfs,mathtools}
\usepackage{stmaryrd}
\usepackage[all]{xy}

\usepackage{epsfig}

\usepackage[all]{xy}
\xyoption{poly}
\usepackage{fancyhdr}
\usepackage{wrapfig}
\usepackage{epsfig}
\usepackage{thm-restate}
\usepackage[pdftex]{hyperref}

\usepackage{bm}

\usepackage{enumerate}
\usepackage{faktor}
\usepackage{tikz-cd}
\usepackage{nicefrac}
\usepackage{mdframed}
\usepackage{cleveref}
\usepackage{soul}

\usepackage{svg}

\DeclareDocumentCommand{\mathdef}{mO{0}m}{%
  \expandafter\let\csname old\string#1\endcsname=#1
  \expandafter\newcommand\csname new\string#1\endcsname[#2]{#3}
  \DeclareRobustCommand#1{%
    \ifmmode
      \expandafter\let\expandafter\next\csname new\string#1\endcsname
    \else
      \expandafter\let\expandafter\next\csname old\string#1\endcsname
    \fi
    \next
  }%
}

\newcommand{\mf}[1]{\mathfrak{#1}} 
\newcommand{\mb}[1]{\mathbb{#1}} 
\newcommand{\mc}[1]{\mathcal{#1}} 
\newcommand{\ms}[1]{\mathscr{#1}} 
\mathdef{\set}[1]{{\left\{#1\right\}}} 

\mathdef{\a}{\alpha}
\mathdef{\b}{\beta}
\mathdef{\g}{\gamma}
\mathdef{\d}{\delta}
\mathdef{\e}{\epsilon}
\mathdef{\z}{\zeta}
\mathdef{\h}{\eta}
\mathdef{\th}{\theta}
\mathdef{\i}{\iota}
\mathdef{\k}{\kappa}
\mathdef{\l}{\lambda}
\mathdef{\m}{\mu}
\mathdef{\n}{\nu}
\mathdef{\x}{\xi}
\mathdef{\o}{\omicron}
\mathdef{\p}{\pi}
\mathdef{\r}{\rho}
\mathdef{\s}{{\sigma}}
\mathdef{\t}{\tau}
\mathdef{\y}{\upsilon}
\mathdef{\f}{\phi}
\mathdef{\vf}{\varphi} 
\mathdef{\c}{\chi}
\mathdef{\ps}{\psi}
\mathdef{\w}{\omega}

\mathdef{\G}{\Gamma}
\mathdef{\S}{\Sigma}

\mathdef{\1}{\mb{1}} 
\mathdef{\A}{\mb{A}}
\mathdef{\B}{\mb{B}}
\mathdef{\C}{\mb{C}}
\mathdef{\D}{\mb{D}}
\mathdef{\E}{\mb{E}}
\mathdef{\F}{\mb{F}}
\mathdef{\H}{\mb{H}}
\mathdef{\I}{\mb{I}}
\mathdef{\J}{\mb{J}}
\mathdef{\K}{\mb{K}}
\mathdef{\L}{\mb{L}}
\mathdef{\M}{\mb{M}}
\mathdef{\N}{\mb{N}}
\mathdef{\O}{\mb{O}}
\mathdef{\P}{\mb{P}}
\mathdef{\Q}{\mb{Q}}
\mathdef{\R}{\mb{R}}
\mathdef{\T}{\mb{T}}
\mathdef{\U}{\mb{U}}
\mathdef{\V}{\mb{V}}
\mathdef{\W}{\mb{W}}
\mathdef{\X}{\mb{X}}
\mathdef{\Y}{\mb{Y}}
\mathdef{\Z}{\mb{Z}}

\mathdef{\cA}{\mc{A}}
\mathdef{\cB}{\mc{B}}
\mathdef{\cC}{\mc{C}}
\mathdef{\cD}{\mc{D}}
\mathdef{\cE}{\mc{E}}
\mathdef{\cF}{\mc{F}}
\mathdef{\cG}{\mc{G}}
\mathdef{\cH}{\mc{H}}
\mathdef{\cI}{\mc{I}}
\mathdef{\cJ}{\mc{J}}
\mathdef{\cK}{\mc{K}}
\mathdef{\cL}{\mc{L}}
\mathdef{\cM}{\mc{M}}
\mathdef{\cN}{\mc{N}}
\mathdef{\cO}{\mc{O}}
\mathdef{\cP}{\mc{P}}
\mathdef{\cQ}{\mc{Q}}
\mathdef{\cR}{\mc{R}}
\mathdef{\cS}{\mc{S}}
\mathdef{\cT}{\mc{T}}
\mathdef{\cU}{\mc{U}}
\mathdef{\cV}{\mc{V}}
\mathdef{\cW}{\mc{W}}
\mathdef{\cX}{\mc{X}}
\mathdef{\cY}{\mc{Y}}
\mathdef{\cZ}{\mc{Z}}

\mathdef{\sA}{\ms{A}}
\mathdef{\sB}{\ms{B}}
\mathdef{\sC}{\ms{C}}
\mathdef{\sD}{\ms{D}}
\mathdef{\sE}{\ms{E}}
\mathdef{\sF}{\ms{F}}
\mathdef{\sG}{\ms{G}}
\mathdef{\sH}{\ms{H}}
\mathdef{\sI}{\ms{I}}
\mathdef{\sJ}{\ms{J}}
\mathdef{\sK}{\ms{K}}
\mathdef{\sL}{\ms{L}}
\mathdef{\sM}{\ms{M}}
\mathdef{\sN}{\ms{N}}
\mathdef{\sO}{\ms{O}}
\mathdef{\sP}{\ms{P}}
\mathdef{\sQ}{\ms{Q}}
\mathdef{\sR}{\ms{R}}
\mathdef{\sS}{\ms{S}}
\mathdef{\sT}{\ms{T}}
\mathdef{\sU}{\ms{U}}
\mathdef{\sV}{\ms{V}}
\mathdef{\sW}{\ms{W}}
\mathdef{\sX}{\ms{X}}
\mathdef{\sY}{\ms{Y}}
\mathdef{\sZ}{\ms{Z}}

\mathdef{\sa}{\ms{a}}
\mathdef{\sb}{\ms{b}}
\mathdef{\sd}{\ms{d}}
\mathdef{\se}{\ms{e}}
\mathdef{\sf}{\ms{f}}
\mathdef{\sg}{\ms{g}}
\mathdef{\sh}{\ms{h}}
\mathdef{\si}{\ms{i}}
\mathdef{\sj}{\ms{j}}
\mathdef{\sk}{\ms{k}}
\mathdef{\sl}{\ms{l}}
\mathdef{\sm}{\ms{m}}
\mathdef{\sn}{\ms{n}}
\mathdef{\sp}{\ms{p}}
\mathdef{\sq}{\ms{q}}
\mathdef{\sr}{\ms{r}}
\mathdef{\ss}{\ms{s}}
\mathdef{\st}{\ms{t}}
\mathdef{\su}{\ms{u}}
\mathdef{\sv}{\ms{v}}
\mathdef{\sw}{\ms{w}}
\mathdef{\sx}{\ms{x}}
\mathdef{\sy}{\ms{y}}
\mathdef{\sz}{\ms{z}}

\mathdef{\fA}{\mf{A}}
\mathdef{\fB}{\mf{B}}
\mathdef{\fC}{\mf{C}}
\mathdef{\fD}{\mf{D}}
\mathdef{\fE}{\mf{E}}
\mathdef{\fF}{\mf{F}}
\mathdef{\fG}{\mf{G}}
\mathdef{\fH}{\mf{H}}
\mathdef{\fI}{\mf{I}}
\mathdef{\fJ}{\mf{J}}
\mathdef{\fK}{\mf{K}}
\mathdef{\fL}{\mf{L}}
\mathdef{\fM}{\mf{M}}
\mathdef{\fN}{\mf{N}}
\mathdef{\fO}{\mf{O}}
\mathdef{\fP}{\mf{P}}
\mathdef{\fQ}{\mf{Q}}
\mathdef{\fR}{\mf{R}}
\mathdef{\fS}{\mf{S}}
\mathdef{\fT}{\mf{T}}
\mathdef{\fU}{\mf{U}}
\mathdef{\fV}{\mf{V}}
\mathdef{\fW}{\mf{W}}
\mathdef{\fX}{\mf{X}}
\mathdef{\fY}{\mf{Y}}
\mathdef{\fZ}{\mf{Z}}
\mathdef{\fa}{\mf{a}}
\mathdef{\fb}{\mf{b}}
\mathdef{\fc}{\mf{c}}
\mathdef{\fd}{\mf{d}}
\mathdef{\fe}{\mf{e}}
\mathdef{\ff}{\mf{f}}
\mathdef{\fg}{\mf{g}}
\mathdef{\fh}{\mf{h}}
\mathdef{\ffi}{\mf{i}} 
\mathdef{\fj}{\mf{j}}
\mathdef{\fk}{\mf{k}}
\mathdef{\fl}{\mf{l}}
\mathdef{\fm}{\mf{m}}
\mathdef{\fn}{\mf{n}}
\mathdef{\fo}{\mf{o}}
\mathdef{\fp}{\mf{p}}
\mathdef{\fq}{\mf{q}}
\mathdef{\fs}{\mf{s}}
\mathdef{\ft}{\mf{t}}
\mathdef{\fu}{\mf{u}}
\mathdef{\fv}{\mf{v}}
\mathdef{\fw}{\mf{w}}
\mathdef{\fx}{\mf{x}}
\mathdef{\fy}{\mf{y}}
\mathdef{\fz}{\mf{z}}

\mathdef{\bA}{\bm{A}}
\mathdef{\bB}{\bm{B}}
\mathdef{\bC}{\bm{C}}
\mathdef{\bD}{\bm{D}}
\mathdef{\bE}{\bm{E}}
\mathdef{\bF}{\bm{F}}
\mathdef{\bG}{\bm{G}}
\mathdef{\bH}{\bm{H}}
\mathdef{\bI}{\bm{I}}
\mathdef{\bJ}{\bm{J}}
\mathdef{\bK}{\bm{K}}
\mathdef{\bL}{\bm{L}}
\mathdef{\bM}{\bm{M}}
\mathdef{\bN}{\bm{N}}
\mathdef{\bO}{\bm{O}}
\mathdef{\bP}{\bm{P}}
\mathdef{\bQ}{\bm{Q}}
\mathdef{\bR}{\bm{R}}
\mathdef{\bS}{\bm{S}}
\mathdef{\bT}{\bm{T}}
\mathdef{\bU}{\bm{U}}
\mathdef{\bV}{\bm{V}}
\mathdef{\bW}{\bm{W}}
\mathdef{\bX}{\bm{X}}
\mathdef{\bY}{\bm{Y}}
\mathdef{\bZ}{\bm{Z}}
\mathdef{\ba}{\bm{a}}
\mathdef{\bb}{\bm{b}}
\mathdef{\bc}{\bm{c}}
\mathdef{\bd}{\bm{d}}
\mathdef{\be}{\bm{e}}
\mathdef{\bf}{\bm{f}}
\mathdef{\bg}{\bm{g}}
\mathdef{\bh}{\bm{h}}
\mathdef{\bi}{\bm{i}}
\mathdef{\bj}{\bm{j}}
\mathdef{\bk}{\bm{k}}
\mathdef{\bl}{\bm{l}}
\mathdef{\bbm}{\bm{m}}
\mathdef{\bn}{\bm{n}}
\mathdef{\bo}{\bm{o}}
\mathdef{\bp}{\bm{p}}
\mathdef{\bq}{\bm{q}}
\mathdef{\br}{\bm{r}}
\mathdef{\bs}{\bm{s}}
\mathdef{\bt}{\bm{t}}
\mathdef{\bu}{\bm{u}}
\mathdef{\bv}{\bm{v}}
\mathdef{\bw}{\bm{w}}
\mathdef{\bx}{\bm{x}}
\mathdef{\by}{\bm{y}}
\mathdef{\bz}{\bm{z}}

\mathdef{\RP}[1]{\R P^{#1}}
\mathdef{\CP}[1]{\C P^{#1}}
\mathdef{\GL}[2]{{\text{GL}(#1, #2)}}
\mathdef{\SL}[2]{{\text{SL}(#1, #2)}}
\mathdef{\SO}[1]{{\text{SO}(#1)}}
\mathdef{\SU}[1]{{\text{SU}(#1)}}
\mathdef{\OO}[1]{{\text{O}(#1)}}
\mathdef{\UU}[1]{{\text{U}(#1)}}
\mathdef{\Sp}[1]{{\text{Sp}(#1)}}
\mathdef{\PGL}[2]{{\text{PGL}(#1, #2)}}
\mathdef{\PSL}[2]{{\text{PSL}(#1, #2)}}

\mathdef{\QF}{\mc{QF}}
\mathdef{\ML}{\mc{ML}}
\mathdef{\QT}{\mc{QT}}
\mathdef{\PT}{\mc{PT}}

\mathdef{\del}{\nabla} 
\mathdef{\transverse}{\pitchfork}
\mathdef{\compl}{\backslash} 
\mathdef{\es}{\emptyset} 
\mathdef{\Wedge}{\bigwedge}
\mathdef{\Oplus}{\bigoplus}
\mathdef{\Otimes}{\bigotimes}
\mathdef{\eval}{\bigg|} 

\mathdef{\union}{\cup}
\mathdef{\intersect}{\cap}
\mathdef{\Union}{\bigcup}
\mathdef{\Intersect}{\bigcap}
\mathdef{\disunion}{\sqcup}
\mathdef{\hequiv}{\simeq}
\mathdef{\compose}{\circ}

\mathdef{\lquot}{\setminus}
\mathdef{\rquot}{/}

\newcommand{\lrp}[1]{\left(#1\right)} 
\newcommand{\lrb}[1]{\left[#1\right]} 
 
\newcommand{\ra}{\rightarrow} 

\DeclareMathOperator{\injrad}{injrad} 

\DeclareMathOperator{\vol}{vol} 

\DeclareMathOperator{\WP}{WP}

\DeclareMathOperator{\hol}{hol}
\DeclareMathOperator{\hyp}{hyp}

\DeclareMathOperator{\vis}{vis}

\advance \topmargin by -\headheight
\advance \topmargin by -\headsep
\evensidemargin \oddsidemargin
  \def\C{{\mathbb{C}}}  \def\E{{\mathbb{E}}} \def\F{{\mathbb{F}}}  \def\H{{\mathbb{H}}}      \def\N{{\mathbb{N}}} \def\OO{{\mathbb{O}}} \def\P{{\mathbb{P}}} \def\Q{{\mathbb{Q}}} \def\R{{\mathbb{R}}}   \def\UU{{\mathbb{U}}}     \def\Z{{\mathbb{Z}}}

\def\bA{{\bar{A}}} \def\bB{{\bar{B}}} \def\bC{{\bar{C}}} \def\bD{{\bar{D}}} \def\bE{{\bar{E}}} \def\bF{{\bar{F}}} \def\bG{{\bar{G}}} \def\bH{{\bar{H}}} \def\bI{{\bar{I}}} \def\bJ{{\bar{J}}} \def\bK{{\bar{K}}} \def\bL{{\bar{L}}} \def\bM{{\bar{M}}} \def\bN{{\bar{N}}} \def\bO{{\bar{O}}} \def\bP{{\bar{P}}} \def\bQ{{\bar{Q}}} \def\bR{{\bar{R}}} \def\bS{{\bar{S}}} \def\bT{{\bar{T}}} \def\bU{{\bar{U}}} \def\bV{{\bar{V}}} \def\bW{{\bar{W}}} \def\bX{{\bar{X}}} \def\bY{{\bar{Y}}} \def\bZ{{\bar{Z}}}

\def\ba{{\bar{a}}} \def\bb{{\bar{b}}} \def\bc{{\bar{c}}} \def\bd{{\bar{d}}} \def\be{{\bar{e}}}  \def\bg{{\bar{g}}} \def\bh{{\bar{h}}} \def\bi{{\bar{i}}} \def\bj{{\bar{j}}} \def\bk{{\bar{k}}} \def\bl{{\bar{l}}} \def\bm{{\bar{m}}} \def\bn{{\bar{n}}} \def\bo{{\bar{o}}} \def\bp{{\bar{p}}} \def\bq{{\bar{q}}} \def\br{{\bar{r}}} \def\bs{{\bar{s}}} \def\bt{{\bar{t}}} \def\bu{{\bar{u}}} \def\bv{{\bar{v}}} \def\bw{{\bar{w}}} \def\bx{{\bar{x}}} \def\by{{\bar{y}}} \def\bz{{\bar{z}}}

\def\cA{{\mathcal{A}}} \def\cB{{\mathcal{B}}} \def\cC{{\mathcal{C}}} \def\cD{{\mathcal{D}}} \def\cE{{\mathcal{E}}} \def\cF{{\mathcal{F}}} \def\cG{{\mathcal{G}}} \def\cH{{\mathcal{H}}} \def\cI{{\mathcal{I}}} \def\cJ{{\mathcal{J}}} \def\cK{{\mathcal{K}}} \def\cL{{\mathcal{L}}} \def\cM{{\mathcal{M}}} \def\cN{{\mathcal{N}}} \def\cO{{\mathcal{O}}} \def\cP{{\mathcal{P}}} \def\cQ{{\mathcal{Q}}} \def\cR{{\mathcal{R}}} \def\cS{{\mathcal{S}}} \def\cT{{\mathcal{T}}} \def\cU{{\mathcal{U}}} \def\cV{{\mathcal{V}}} \def\cW{{\mathcal{W}}} \def\cX{{\mathcal{X}}} \def\cY{{\mathcal{Y}}} \def\cZ{{\mathcal{Z}}}

\def\fA{{\mathfrak{A}}} \def\fB{{\mathfrak{B}}} \def\fC{{\mathfrak{C}}} \def\fD{{\mathfrak{D}}} \def\fE{{\mathfrak{E}}} \def\fF{{\mathfrak{F}}} \def\fG{{\mathfrak{G}}} \def\fH{{\mathfrak{H}}} \def\fI{{\mathfrak{I}}} \def\fJ{{\mathfrak{J}}} \def\fK{{\mathfrak{K}}} \def\fL{{\mathfrak{L}}} \def\fM{{\mathfrak{M}}} \def\fN{{\mathfrak{N}}} \def\fO{{\mathfrak{O}}} \def\fP{{\mathfrak{P}}} \def\fQ{{\mathfrak{Q}}} \def\fR{{\mathfrak{R}}} \def\fS{{\mathfrak{S}}} \def\fT{{\mathfrak{T}}} \def\fU{{\mathfrak{U}}} \def\fV{{\mathfrak{V}}} \def\fW{{\mathfrak{W}}} \def\fX{{\mathfrak{X}}} \def\fY{{\mathfrak{Y}}} \def\fZ{{\mathfrak{Z}}}

\def\fa{{\mathfrak{a}}} \def\fb{{\mathfrak{b}}} \def\fc{{\mathfrak{c}}} \def\fd{{\mathfrak{d}}} \def\fe{{\mathfrak{e}}} \def\ff{{\mathfrak{f}}} \def\fg{{\mathfrak{g}}} \def\fh{{\mathfrak{h}}}  \def\fj{{\mathfrak{j}}} \def\fk{{\mathfrak{k}}} \def\fl{{\mathfrak{l}}} \def\fm{{\mathfrak{m}}} \def\fn{{\mathfrak{n}}} \def\fo{{\mathfrak{o}}} \def\fp{{\mathfrak{p}}} \def\fq{{\mathfrak{q}}}  \def\fs{{\mathfrak{s}}} \def\ft{{\mathfrak{t}}} \def\fu{{\mathfrak{u}}} \def\fv{{\mathfrak{v}}} \def\fw{{\mathfrak{w}}} \def\fx{{\mathfrak{x}}} \def\fy{{\mathfrak{y}}} \def\fz{{\mathfrak{z}}}

\def\sA{{\mathscr{A}}} \def\sB{{\mathscr{B}}} \def\sC{{\mathscr{C}}} \def\sD{{\mathscr{D}}} \def\sE{{\mathscr{E}}} \def\sF{{\mathscr{F}}} \def\sG{{\mathscr{G}}} \def\sH{{\mathscr{H}}} \def\sI{{\mathscr{I}}} \def\sJ{{\mathscr{J}}} \def\sK{{\mathscr{K}}} \def\sL{{\mathscr{L}}} \def\sM{{\mathscr{M}}} \def\sN{{\mathscr{N}}} \def\sO{{\mathscr{O}}} \def\sP{{\mathscr{P}}} \def\sQ{{\mathscr{Q}}} \def\sR{{\mathscr{R}}} \def\sS{{\mathscr{S}}} \def\sT{{\mathscr{T}}} \def\sU{{\mathscr{U}}} \def\sV{{\mathscr{V}}} \def\sW{{\mathscr{W}}} \def\sX{{\mathscr{X}}} \def\sY{{\mathscr{Y}}} \def\sZ{{\mathscr{Z}}}

       \def\th{{\tilde{h}}}                  

     \def\vf{{\vec{f}}}                    

\renewcommand\a{\alpha}
\renewcommand\b{\beta}
\renewcommand\d{\delta}
\renewcommand\k{\kappa}
\renewcommand\l{\lambda}

\renewcommand\t{\tau}

\newcommand\ray{{\mathsf{ray}}}

\newcommand\Sector{\mathsf{Sector}}

  \newcommand{\dee}{\mathrm{d}}

\declaretheorem[numberwithin=section, style=plain]{theorem}

\declaretheorem{conjecture, lemma, corollary, proposition, convention}[sibling=theorem, style=plain]
[name=Definition, style=definition, sibling=theorem]
\declaretheorem{remark, example, question, notation}[sibling=definition, style=definition]

[name=Theorem, style=plain]
[sibling=maintheorem, style=plain, name=Proposition]
[name=Definition, style=definition, sibling=definition]
[name=Proposition, style=plain, sibling=proposition]
[name=Lemma, style=plain, sibling=lemma]
[name=Corollary, style=plain, sibling=corollary]
[name=Assumption, style=plain, sibling=theorem]

{\hfill\\\\\textbf{Claim #1:} \textit{#2}\\\textit{Proof}. }
{\hfill\\\\}

\usepackage{color}

\title{Benjamini-Schramm limits of high genus translation surfaces: research announcement}

\author{Lewis Bowen\footnote{supported in part by NSF grant DMS-2154680}, Kasra Rafi\footnote{supported in part by NSERC Discovery grant, RGPIN-05507}, and Hunter Vallejos\footnote{supported in part by NSF grant DMS-1937215}}

\begin{document}

    \maketitle

\begin{abstract}

We prove that the sequence of Masur-Smillie-Veech (MSV) distributed random translation surfaces, with area equal to genus, Benjamini-Schramm converges as genus tends to infinity. This means that for any fixed radius $r>0$, if $X_g$ is an MSV-distributed random translation surface with area $g$ and genus $g$, and $o$ is a uniformly random point in $X_g$, then the radius-$r$ neighborhood of $o$ in $X_g$, as a pointed measured metric space, converges in distribution to the radius $r$ neighborhood of the root in a Poisson translation plane, which is a random pointed surface we introduce here. Along the way, we obtain bounds on statistical local geometric properties of translation surfaces, such as the probability that the random point $o$ has injectivity radius at most $r$, which may be of independent interest.
\end{abstract}

\section{Introduction}
Benjamini-Schramm convergence is a notion of convergence for sequences of (random) finite graphs, finite-volume manifolds, and, more generally, measured metric spaces \cite{benjamini_recurrence_2001,abert_unimodular_2022,MR4092421}. A sequence $(X_i)_i$ of random metric spaces, each equipped with a finite measure, is said to Benjamini-Schramm converge to a random pointed metric space $(X_\infty, o_\infty)$ if, when $o_i$ is a random point in $X_i$, the law of $(X_i, o_i)$ converges to the law of $(X_\infty, o_\infty)$ in the space of Borel probability measures on the space of pointed measured metric spaces. Here, we use the weak topology on the space of measures and the pointed Gromov-Hausdorff-Prokhorov topology on the space of pointed translation surfaces.  Intuitively, Benjamini-Schramm convergence describes the local geometry and structure experienced by a ``typical'' point on $X_i$ as $i \to \infty$.

This notion of convergence naturally generalizes to manifolds endowed with additional structures, such as translation surfaces. A translation surface consists of a compact connected surface $X$, a finite set $\Sigma \subset X$, and a translation structure on $X - \Sigma$. This translation structure is a maximal atlas of charts to $\C$ on $X - \Sigma$ such that the transition maps are translations. It induces a Riemann surface structure on $X -\Sigma$, which extends naturally to the entire surface $X$. Moreover, the differential $\dee z$, being invariant under translations, does not depend on the choice of chart. These local differentials glue together to define a global holomorphic $1$-form $\omega$ on $X$. There is also a locally Euclidean metric on $X-\Sigma$ which defines a notion of Lebesgue measure of $X$. In the complex-analytic language, the measure is given by the area form $\frac i 2 (\w \wedge \bar \w)$.
Points in $\Sigma$ are cone points where the total angle is a multiple of $2\pi$. A cone point 
where the total angle is $2 \pi (k+1)$ is called a singularity of order $k$. 
        
Let $\cH_g$ and $\cH_g^{\hyp}$ denote the spaces of genus $g$ translation surfaces of area $1$ and area $g$, respectively. The notation reflects the fact that a translation surface in $\cH_g^{\hyp}$ has an area comparable to the hyperbolic area of the underlying Riemann surface. 

Due to the work of Masur and Veech \cite{masur_interval_1982, veech_teichmuller_1986}, each space admits a finite Lebesgue-class measure, denoted by $\m_g$ on $\cH_g$ and $\m_g^{\hyp}$ on $\cH_g^{\hyp}$, referred to as the \textit{Masur-Smillie-Veech measures}. There is a natural measure $\m_{g*}^{\hyp}$ defined on the set of pointed translation surfaces $(X, o)$, where $X \in \cH_g^{\hyp}$, given by the disintegration formula:
\[
\dee\m_{g*}^{\hyp}(X, o) = \dee\m_X(o)~\dee\m_g^{\hyp}(X),
\]
where $\m_X$ is the Euclidean area measure on $X$. Since both $\m_X$ and $\m_g^{\hyp}$ are finite, the measure $\m_{g*}^{\hyp}$ is also finite. We define the corresponding probability measures:
\[
\P_{g*}^{\hyp} = \frac{\m_{g*}^{\hyp}}{g \cdot \m_g^{\hyp}(\cH_g^{\hyp})}
\qquad \text{and} \qquad
\P_g = \frac{\m_g}{\m_g(\cH_g)}.
\]
Let $\sT_*$ denote the space of all pointed connected translation surfaces up to isomorphism, where all singularities are of finite order and the set of singularities is discrete. The space $\sT_*$ can be endowed with a complete, separable metric using the Gromov-Hausdorff-Prokhorov topology (see \cite{abraham_note_2013} and \cite{MR4092421}). We view $\P_{g*}^{\hyp}$ as a measure on this space.

Our main theorem is:

\begin{theorem}\label{T:main}
The sequence $\set{\P_{g*}^{\hyp}}_{g=2}^\infty$ converges weakly to the distribution of the Poisson translation plane with intensity $4$.
\label{thm:main-theorem}
\end{theorem}

In essence, Theorem \ref{thm:main-theorem} states that in a typical high-genus translation surface of area $g$, a typical point perceives no topology, and the ``visible'' singularities are distributed according to a Poisson point process with intensity $4$. We provide a more detailed description of Poisson translation planes below.

Note that an intensity of $4$ is expected. To see this, define \( N_r : X \to \mathbb{N} \), where $ N_r(o)$ is the number of singularities visible from $o$ that are at most distance $r$ from $o$. If Theorem \ref{thm:main-theorem} holds with intensity $\lambda$, then asymptotically and on average we have, 
\begin{equation} \label{Eq:N}
N_r \sim \lambda \cdot 2\pi r^2,
\end{equation}  
as the generic point is a regular point.  On the other hand,  
\[
\int_X N_r = \sum_{\sigma \in \Sigma} \, (\text{area of the visible ball of radius $r$ around $\sigma$}). 
\]  
Assuming that the ball of radius \( r \) around \( \sigma \) is generically planar, the right-hand side of this equation becomes \( (2g - 2) 4\pi r^2 \) (There are $(2g-2)$ singular points on a surface of  genus $g$).  
Using \eqref{Eq:N} and the fact that the area of \( X \) is \( g \), we have:  
\[
g \cdot \lambda \cdot  2\pi r^2 \sim (2g - 2) 4\pi r^2.  
\]  
That is, as $g \to \infty$, we have  \( \lambda \to 4 \).

\subsection{What is a Poisson translation plane?}
We begin by describing some properties of a Poisson translation plane. It is a random pointed translation surface $(P, o)$ equipped with a holomorphic $1$-form $\omega$, and it is conformally equivalent to the hyperbolic plane. With probability $1$, the point $o$ is regular, and all singularities are of order one.
  
If $\gamma$ is any path in $P$, its holonomy is the complex number $\hol(\gamma) = \int_\gamma \omega$.  
A point $y \in P$ is said to be \textbf{visible} from a point $x \in P$ if there exists a geodesic from $x$ to $y$ that does not pass through any singularities in its interior.  
The \textbf{visible set} of $x$ is the set of all points $y \in P$ that are visible from $x$. The \textbf{visible singularities} of $x$ are defined similarly. 
    
The structure of $P$ is easiest to understand from the perspective of the base point $o \in P$. Consider the collection $\{\hol(\gamma)\}_\gamma$, where $\gamma$ ranges over all geodesics from $o$ to a visible singularity. In a Poisson translation plane, this collection of complex numbers forms a Poisson point process $\Pi_o$ on $\C$. Each singularity $\s$ visible from $o$ is of order $1$, with a cone angle of $4\pi$. 

At each such singularity $\s$, there is a sector $\Sector_{\s} \subset P$ of angle $2\pi$ comprised of points visible from $\s$ but not from $o$. Holonomy maps this sector to a dense open subset $\hol(\Sector_{\s})\subset \C$ of the complex plane. The set $\{\hol(\gamma)\}_\gamma$ of holonomies of geodesics $\gamma$ connecting $\s$ to a singularity $\t \in \Sector_{\s}$ also forms a Poisson point process $\Pi_{\s}$ on $\C$. Continuing in this manner, for each $\t \in \Sector_{\s}$, there exists a sector $\Sector_{\t}$ of angle $2\pi$ at $\t$, consisting of points not visible from $\s$ but visible from $\t$. The singularities in $\Sector_{\t}$ visible from $\t$ also form a Poisson point process. 

In the Poisson translation plane, all of these Poisson point processes are jointly independent and share the same intensity. This common intensity defines the intensity of the Poisson translation plane.

    One can construct a Poisson translation plane of intensity $\l > 0$ in the following way. Let $\C_o$ be a copy of the complex plane and set $o$ to be $0 \in \C_o$. Sample a Poisson point process $\Pi_o$ of intensity $\l$ in $\C_o$. At each point $x \in \Pi_o$, we make a cut along the ray $\ray_x=[1,\infty]\cdot x$ and take the path-metric completion. We call this a slit-plane. Let $\C_x$ be a copy of the complex plane. Cut $\C_x$ along the ray $[0,\infty]\cdot x$ and take the path-metric completion. Now glue this to the slit-plane by identifying boundary-components in a holonomy-preserving manner. We have now constructed the depth $1$ approximation to the Poisson translation plane. 

     Now, on each $\C_x$, we take another Poisson point process $\Pi_x$ of intensity $\l$, all jointly independent, and perform the exact same procedure. Continuing this process forever ends with a Poisson translation plane of intensity $\l$. This is made more precise in 
     our upcoming paper \cite{BRV}. 

\subsection{Scaling}
In the literature, particularly in Teichmüller dynamics, it is common to focus on \textit{unit-area} translation surfaces. The number of singularities (counted with multiplicity) on such a surface is of order $g$, the genus. As $g \to \infty$, the number of singularities expected in a ball of radius $r$ around a uniformly random basepoint in a unit-area translation surface grows without bound. In particular, by local compactness, singularities will accumulate in any Benjamini-Schramm limit. While it is possible to study such limits, they are not the focus of this manuscript.

Instead, we scale the metric on unit-area translation surfaces by $\sqrt{g}$, giving surfaces with area $g$. This scaling allows us to translate fluently between unit-area and area $g$ translation surfaces. For instance, if $Y = \sqrt{g} \cdot X$ and $X$ is a unit-area translation surface, then there is a natural bijection between $Y$ and $X$. The radius $R$ ball in $Y$ corresponds to a ball of radius $\frac{R}{\sqrt{g}}$ in $X$. Scaling by $\sqrt{g}$ ensures that, at least intuitively, the expected number of singularities in a ball of radius $R$ around a typical basepoint in $Y$ is bounded by a function of $R$, and this number is positive. On the other hand, if the area scales super-linearly with genus, then the Benjamini-Schramm limit is almost surely $\mathbb{C}$. This means that with high probability,  there are no singularities near a randomly chosen basepoint. 

Thus, the critical case occurs when the area of the translation surface grows linearly with the genus. We choose area $g$ for convenience, but a qualitatively equivalent limit can be obtained for any linearly growing area in $g$. Indeed, suppose $(X_g,o_g)$ is sampled randomly from $\P^{\hyp}_{g*}$ as in Theorem \ref{T:main}. Now let $s X_g$ be the same surface as $X_g$ with its metric multiplied by $s$. Then the area of $sX_g$ is $s^2g$. A ball of radius $R$ in $X_g$ corresponds with a ball of radius $sR$ in $sX_g$. The average number of singularities visible from the root in the ball of radius $R$ centered at $o_g$ in $X_g$ limits on $4\pi R^2$ as $g\to\infty$, because $(X_g,o_g)$ BS-converges to a Poisson translation plane of intensity $4$. So the average number of singularities visible from the root in the ball of radius $sR$ centered at $o_g$ in $sX_g$ also limits on $4\pi R^2$. This implies that $(sX_g,o_g)$ BS-limits to a Poisson translation plane with intensity $\frac{4\pi R^2}{\pi(sR)^2}=4 /s^2$.

For example, one could set the area of the translation surface equal to the hyperbolic area $2\pi(2g-2)$. In this case, the intensity of the Poisson translation plane in Theorem \ref{thm:main-theorem} would be $\lambda = \frac{1}{{\pi}}$.

    \subsection{Estimates for unit-area translation surfaces}
    \label{subsec:estimates-for-unit-area-translation-surfaces}
    The proof of Theorem \ref{thm:main-theorem} is done in three main steps. First, we prove that $\P_{g*}^{\hyp}$-random translation surfaces do not have too many singularities near the basepoint; we then prove planarity of the limit. These two results together imply tightness of $\set{\P_{g*}^{\hyp}}_{g \geq 2}$. By Prokhorov's Theorem, this implies existence of subsequential limits. We then must prove, for example, that the number of singularities visible from the root in a ball of radius $R$ converges in distribution to a Poisson random variable. Moreover, the holonomies of the shortest paths from the root to these singularities are uniformly distributed. This is only an example because there may be singularities in the radius $R$ ball that are not visible to the root and we also need to describe the distribution of their positions relative to the root.
    
    The proof therefore requires three kinds of estimates on unit-area translation surfaces. We state below the versions of these estimates when the basepoint is conditioned to be a fixed (label of a) singularity. Rigorous statements involving random basepoints sampled according to Euclidean area appear in \cite{BRV}.

    The following estimate deals with the case of too many singularities near the basepoint. 
    
    \begin{restatable}{proposition}{NumSing}
      There exists a universal constant $C > 0$ such that the following holds. Let $(X, \w) \sim \P_g$ be a random translation surface of genus $g$ and area 1 distributed according to Masur-Smillie-Veech measure. Let $\s, \t$ be fixed (labels of) singularities in $X$. Let $r >0$ and $d_X$ denote the flat metric on $X$. Then
      \[
        \P_g[d_X(\s, \t)\leq r] \leq \frac{C}{g} \exp(4r \sqrt{\pi g}).
      \]
      \label{prop:NumSing}
    \end{restatable}
    
     The next estimate deals with the probability that the injectivity radius at a uniformly random point is small; this is necessary to prove planarity of the Benjamini-Schramm limit as well as tightness. It is easier to first prove such for singularities. For a (label of a) singularity $\s$ in a translation surface $(X, \w)$, we let $\injrad_X(\s)$ denote the supremum of radii $r$ such that the metric ball of radius $r$ about $\s$ is simply connected.

    \begin{restatable}{proposition}{NumSimpleGeod}
      Let $(X, \w) \sim \P_g$ be a random unit-area translation surface of genus $g$ distributed according to Masur-Smillie-Veech measure. Let $\s$ be a fixed (label of a) singularity in $X$ and $R > 0$. Then
      \[
        \P_g\lrb{\injrad_X(\s) < \frac R {\sqrt g}} = o_R(1).
      \]
      as $g \to \infty$.
      \label{prop:NumSimpleGeod}
    \end{restatable}
    The notation $o_R(1)$ means that 
    \[
    \P\lrb{\injrad_X(o) < \frac R {\sqrt g}} = f(R, g)
    \]
    for some function $f$ such that
    \[
    \lim_{g \to \infty} f(R, g) = 0
    \]
    for all $R > 0$.

    We lastly compute the limiting distribution of visible singularities. 
    In \cite{BRV}, we also show joint independence between multiple distinct singularities. 

\begin{restatable}{proposition}{NumVisSing}
Let $(X, \omega) \sim \P_g$ be a Masur-Smillie-Veech random unit-area translation surface of genus $g$, and let $\s$ be a fixed singularity in $X$. Let $\Sigma_{\vis}(X, \s, r)$ denote the set of singularities $\t$ such that $\tau$ is visible from $\s$ in the ball of radius $r$ centered at $\s$.  Then for every $R > 0$ and $k \in \mathbb{N} \cup \{0\}$,
\[
  \P\left(\#\Sigma_{\vis}\left(X, \s, \frac{R}{\sqrt{g}}\right) = k\right) = \binom{2g-3}{k} \left( \frac{4\pi R^2}{g} \right)^k \left( 1 - \frac{4\pi R^2}{g} \right)^{2g - k - 3} + o_{k, R}(1)
\]
as $g \to \infty$.
\label{prop:NumVisSing}
\end{restatable}

\begin{corollary}
With the notation of Proposition \ref{prop:NumVisSing}, we have
\[
  \lim_{g \to \infty} \P\left(\#\Sigma_{\vis}\left(X, \s, \frac{R}{\sqrt{g}}\right) = k\right) = \frac{\left(8 \pi R^2\right)^k}{k!} e^{-8\pi R^2}.
\]
In other words, the random variables $N_g = \#\Sigma_{\vis}\left(X, \s, \frac{R}{\sqrt{g}}\right)$ converge in distribution to a Poisson random variable with mean $8 \pi R^2$ as $g \to \infty$.
\label{cor:get-poisson-distr-from-counting-vis-sings}
\end{corollary}

\begin{proof}[Proof of Corollary \ref{cor:get-poisson-distr-from-counting-vis-sings}]
Sending $g \to \infty$, we apply the approximation
\[
  \binom{2g - 3}{k} \sim \frac{(2g)^k}{k!}
\]
to the estimate in Proposition \ref{prop:NumVisSing}, obtaining
\[
  \P\left(\#\Sigma_{\vis}\left(X, \s, \frac{R}{\sqrt{g}}\right) = k\right) \sim \frac{2^k g^k}{k!} \cdot \frac{\left(4 \pi R^2\right)^k}{g^k} \left(1 - \frac{4\pi R^2}{g}\right)^{2g - k - 3}.
\]
Taking the limit as $g \to \infty$ completes the proof.
\end{proof}

    \subsection{Related literature}
    Benjamini-Schramm convergence was introduced by Itai Benjamini and Oded Schramm to study random
    connected planar graphs \cite{benjamini_recurrence_2001}. Such a notion admits natural generalizations to metric spaces and
    manifolds with additional structure \cite{khezeli2023unimodularrandommeasuredmetric}. One can understand the notion of Benjamini-Schramm limit to
    be an answer to the question ``what it is like to live at a typical point in the space?'' Some examples of Benjamini-Schramm convergence include:
    \begin{enumerate}
    \item If $G_n$ is a uniformly random $d$-regular graph with $n$ vertices and $r_n \in G_n$ is
      a root vertex chosen at random, then the pairs $(G_n, r_n)$ Benjamini-Schramm
      converge to a $d$-regular tree. Indeed, one can show that for any $t$ larger than 2, the
      number of cycles of length $t$ in $G_n$ converges to a Poisson random variable with parameter
      $\frac{(d-1)^t}{2t}$. However, there are $n$ vertices, so most vertices do not lie on a cycle of length $t$
      as $n$ tends to infinity with $t$ held fixed. See \cite[Theorem 2.5]{lamb_models_1999} for more details.
      \item The Benjamini-Schramm limit for Weil-Petersson random high genus hyperbolic surfaces was
      proven to be the hyperbolic plane by Monk \cite{monk_benjaminischramm_2022} using techniques of 
      Mirzakhani \cite{mirzakhani_growth_2013}. In particular, for any $R>0$, most points in a generic hyperbolic 
      surface $X$ of high genus are far from any simple closed curve of length $\le R$.
      \item In \cite{MR4199439}, it shown that if, for every genus $g\ge 2$, one chooses uniformly at random a triangulation of a surface of genus $g$ with $n$ triangles (where $n$ is proportional to $g$) then these graphs Benjamini-Schramm converge to a Planar Stochastic Hyperbolic Triangulation, which is a random triangulation of the hyperbolic plane introduced in \cite{MR3520011}. This confirms a conjecture of Benjamini and Curien. This result was inspirational to us in the early stages of this project.

    \item In \cite{MR}, similar results are obtained regarding the Benjamini-Schramm limits of periodic orbits in homogeneous spaces and the moduli space of translation surfaces.
    \end{enumerate}

    Let us digress on high genus random hyperbolic surfaces and compare it with the current work. The moduli space of hyperbolic surfaces $\cM_{g}$ of genus $g$ admits a finite measure $\m_{\WP}^{g}$ called Weil-Petersson measure. In the mid to late 2000s, Maryam Mirzakhani developed a way to integrate ``geometric'' functions $f: \cM_{g} \ra \R_+$ according to $\m_{\WP}$ --- her famous integration formula \cite{mirzakhani_simple_2006}. The integration formula enabled Mirzakhani to develop recursive relations among the volumes of various moduli spaces of bordered hyperbolic surfaces. While some of these volumes were known via other methods, the recursive relations enabled her to obtain asymptotics for these volumes as genus tends to infinity in \cite{mirzakhani_growth_2013}. Utilizing these asymptotics and her integration formula, she estimated statistical features of high genus random hyperbolic surfaces. For instance, using her results on the number of simple closed curves of bounded length, Mirzakhani proved\footnote{The bound presented here is a corrected version of Mirzakhani's original asymptotic which was provided by Laura Monk \cite{monk_benjaminischramm_2022}.} that for $\m_{WP}^g$-random hyperbolic surfaces $X$,
    \[
    \P_{\m_{\WP}^g}\lrb{\frac{1}{\vol(X)}\vol_X\lrp{\set{x \in X : \injrad(x) < \frac 1 6 \log(g)}}} = O(g^{-1/3}),
    \]
    where $\vol_X$ is the hyperbolic area on $X$. The above implies that as genus tends to infinity, a uniformly random point (with respect to hyperbolic area) has injectivity radius tending to infinity. It follows that if the Benjamini-Schramm limit of hyperbolic surfaces exists, then it is almost surely simply connected.

    One of the crowning jewels of the field is the result of Mirzakhani and Petri \cite{mirzakhani_lengths_2019}, which proved that as $g \ra \infty$, the collection of lengths of the simple closed geodesics on $\m_{\WP}^{g}$-random hyperbolic surface $X$ forms a Poisson point process on $\R_+$. We note that the intensity measure of the Poisson point process is \textit{not} Lebesgue!

    In \cite{Masur-Randecker-Rafi} Masur, Randecker and Rafi 
     obtain analogous results to Mirzakhani and Petri for high genus translation surfaces. They prove that the number of saddle connections of lengths in fixed intervals are independent Poisson random variables as genus tends to infinity. Though similar to the present work, their theorems differ in that they deal with the global situation, whereas we focus entirely on local phenomena. For instance, while Theorem \ref{thm:main-theorem} can be used to prove that the number of saddle connections with lengths in a fixed bounded interval is Poisson distributed locally, our techniques do not readily extend to the global case, as we rely heavily on our planarity result. 
    
    \subsection{Techniques}

    In the current work, our approach is somewhat analogous to Mirzakhani's and Monk's detailed above. There is a large body of work related to translation surfaces that was developed between the 1980s and early 2000s prior to Mirzakhani's breakthroughs. The analogous tool to Mirzakhani's integration formulas in our setting of translation surfaces is Siegel-Veech theory, which allows us to integrate, or compute expected values of, various ``geometric functions'' on moduli spaces of translation surfaces. On the other hand, the conjectured volume asymptotics for strata of translation surfaces were only proven somewhat recently by Aggarwal \cite{aggarwal_large_2020}; he also similarly computed the asymptotic values of Siegel-Veech constants in \cite{aggarwal_large_2019}.    

There are, however, several new challenges in proving the planarity of the Benjamini-Schramm limit for high genus random translation surfaces, challenges that do not arise in the case of hyperbolic surfaces. While Siegel-Veech theory provides effective estimates for the number of closed saddle connections of a given length in a generic translation surface, it does not directly extend to general closed geodesics, particularly those passing through multiple singularities. Developing new ideas to handle these geodesics was a key aspect of our work. A significant portion of the effort in this project was devoted to ruling out short geodesics as part of establishing the planarity of the Benjamini-Schramm limit.

Our approach to proving planarity involves showing that if $X$ contains a short simple closed geodesic near a basepoint $o$, then there exists a nearby surface $X'$, possibly lying in a lower-dimensional stratum, that has a short closed saddle connection, or a pair of homologous saddle connections, near the basepoint $o'$ of $X'$. Using the results of Aggarwal and the Siegel-Veech framework, we show that the volume of translation surfaces containing short simple closed geodesics near the basepoint is asymptotically negligible compared to the volume of the entire stratum as genus tends to infinity.

There is also the case of simple closed geodesics that do not pass through any singularities. These geodesics lie in cylinders foliated by parallel simple closed geodesics and are more straightforward to handle. Here again, Siegel-Veech theory proves to be valuable. Fortunately, this case has essentially been addressed in the work of Masur, Rafi, and Randecker \cite{masur_expected_2022}, which also relies on Aggarwal's volume asymptotics.

It is important to note the results regarding Siegel-Veech constants of \textit{pairs} of saddle connections, such as the work of Athreya, Fairchild, and Masur \cite{athreya_counting_2023}. These are unfortunately inadequate for the problem we seek to solve because we need to deal with chains of $n$ saddle connections, for arbitrary $n$. The central driver of their result on pairs of saddle connections involves the fact that Siegel-Veech transforms of characteristic functions of balls about 0 are in $L^2$. This is a result of Athreya, Cheung, and Masur \cite{athreya_siegel-veech_2019}; on the other hand, it was shown by Athreya and Chaika \cite{athreya_distribution_2012} that the Siegel-Veech transform of a characteristic function of a ball about 0 is \textit{not} in $L^3$, thus crushing hopes that their techniques could be used to study $n$-tuples of saddle connections, for $n > 2$. 

Our central technique is \textit{composing together many surgeries}. These surgeries, considered individually, are not new; they closely resemble those found in the works of Eskin, Masur, and Zorich \cite{eskin_moduli_2003}, Schiffer variations, and rel flow \cite{chaika_ergodic_2024}. However, our approach differs in that we must compose many surgeries together and carefully track additional information about their effects on individual surfaces, such as the holonomies of paths.

We refer to our surgeries as ``star surgeries.'' The geometric perspective afforded by star surgeries is both robust and flexible, enabling us to control complications that would arise if we attempted to directly apply the surgery techniques available in the literature. By composing these star surgeries, we construct measure-preserving maps between specific subsets of strata, all while maintaining careful control of the relevant geometric information. Notably, each proposition in Section \ref{subsec:estimates-for-unit-area-translation-surfaces} relies on a distinct composite surgery tailored to its specific requirements.

\subsection{Acknowledgements}
Lewis and Hunter thank Jon Chaika and Jayadev Athreya for helpful conversations in the early stages of this project.


\end{document}